\documentclass[12pt]{amsart}
\usepackage[latin 1]{inputenc}
\usepackage[T1]{fontenc}
  \usepackage{amsmath,amssymb}
  \usepackage{dsfont}\let\mathbb\mathds
  \usepackage[english,frenchb]{babel}
  \usepackage[all]{xy}
 
\newtheorem{thm}{Théorème}[section]

\newtheorem{propri}[thm]{Propriété}

\newtheorem{conj}[thm]{Conjecture}
\newtheorem{lem}[thm]{Lemme}
\newtheorem{prop}[thm]{Proposition}
\theoremstyle{definition}
\newtheorem{defn}[thm]{Définition}
\theoremstyle{remark}
\newtheorem{rem}[thm]{Remarque}
\numberwithin{equation}{section}
\begin{document}

\title[\'Equivalences entre conjectures de Soergel]{\'Equivalences entre conjectures de Soergel}%
\author{Nicolas Libedinsky}%
 
\address{UFR de Math\'ematiques et Institut de Math\'ematiques de
 Jussieu, Universit\'e Paris 7, 2 place Jussieu, 75251 Paris Cedex 05,
 France}%
\email{libedinsky@math.jussieu.fr}%

\begin{abstract}
  La catégorie  de Soergel $\mathbf{B}_k(V)$ sur un corps $k$ est définie à partir  d'un système de Coxeter
 $(W,\mathcal{S})$ et d'une représentation $k$-linéaire $V $ de $W$. C'est une catégorie de bimodules sur l'algèbre de polynômes sur $V$. C'est aussi une cat\'egorification   de l'alg\`ebre de Hecke de
 $(W,\mathcal{S})$. Dans cet article nous montrons que pour certaines repr\'esentations  $V$ et $V'$ de $W$, la  conjecture de Soergel sur   $\mathbf{B}_k(V')$ est \'equivalente \`a celle sur $\mathbf{B}_k(V)$ . En particulier, quand  $k=\mathbb{R}$, nous pouvons choisir pour $V'$ la repr\'esentation géométrique.
\end{abstract}


\maketitle

\section*{Introduction}

 Dans l'article \cite{S1} de 1992,
 Soergel a cat\'egorifi\'e $\mathcal{H}$, l'algèbre d'Iwahori-Hecke d'un système de Coxeter $(W,\mathcal{S})$. Ceci signifie que si $k$ est un corps infini et $V$ une représentation $k$-linéaire de $W$ satisfaisant certaines propriétés, alors Soergel a construit une cat\'egorie tensorielle $\mathbf{B}_k(V)$ -appelée catégorie de Soergel sur $V$- et un isomorphisme d'anneaux $\varepsilon$
 de $\mathcal{H}$ vers le groupe de Grothendieck scind\'e de
  $\mathbf{B}_k(V)$. 

Il a alors posé une conjecture
 (\ref{cs} ci-dessous) qui donne une correspondance bijective, via $\varepsilon$, entre les éléments de la
 base de Kazhdan-Lusztig et les éléments indécomposables de
 $\mathbf{B}_k(V)$. 

Cette conjecture implique deux résultats majeurs : d'une part, quand $k=\mathbb{R}$, la conjecture de positivité des polynômes de  Kazhdan-Lusztig (voir \cite{S}) et  d'autre part, quand $k$ est de caractéristique positive, une partie de la conjecture de Lusztig portant sur les caractères des
représentations irréductibles de groupes algébriques en caractéristique positive (voir \cite{S2}).

Dans la section 1 de cet article nous donnons les notations et définitions que nous utiliserons dans la suite. Dans la section 2 nous donnons  l'énoncé des trois théorèmes  principaux. Nous choisissons $V$ et $V'$ une représentation et une sous-représentation de $W$ satisfaisant certaines propriétés techniques. Ces propriétés sont satisfaites en particulier quand $k=\mathbb{R}$, $V$ est une des représentations utilisées par Soergel dans sa théorie (une représentation RF) et $V'$ est la représentation géométrique. Le premier théorème (\ref{1}) établit des relations entre les espaces de morphismes de $\mathbf{B}_k(V')$ et ceux de $\mathbf{B}_k(V)$. Le second (\ref{2}) \'enonce une bijection entre les indécomposables de $\mathbf{B}_k(V)$ et ceux de $\mathbf{B}_k(V')$. Le troisième (\ref{ult}) nous donne un isomorphisme entre les groupes de Grothendieck scindés de $\mathbf{B}_k(V)$ et  de $\mathbf{B}_k(V').$ 

Ces deux derniers théorèmes impliquent que la conjecture de Soergel est équivalente pour $\mathbf{B}_k(V)$ et  $\mathbf{B}_k(V')$. En particulier ceci montre que la conjecture de Soergel pour $k=\mathbb{R}$ et $V$ la représentation géométrique implique la conjecture de positivité des polynômes de Kazhdan-Lusztig.

Dans la section 3 nous donnons les outils principaux pour démontrer ces théorèmes : notamment nous travaillerons au niveau des corps de fractions des anneaux de polynômes  pour montrer des isomorphismes entre les espaces de morphismes. Enfin, dans la section 4 nous achevons les démonstrations des théorèmes.

J'aimerais remercier Geordie Williamson pour ses remarques, et Raphaël Rouquier pour son encouragement et ses multiples idées et commentaires.

\section{Définitions}
\subsection{} 
Donnons d'abord quelques définitions.
\begin{defn}
Un syst\`{e}me de Coxeter est un couple $(W,\mathcal{S})$ o\`u $W$
est un groupe et $\mathcal{S}\subseteq W$ une partie
g\'{e}n\'{e}ratrice, tels que $W$ admet une pr\'{e}sentation de
g\'{e}n\'{e}rateurs $s\in \mathcal{S}$ et relations
$(sr)^{m(s,r)}=1$ pour $s,r\in \mathcal{S}$, avec $m(s,s)=1$ ,
$m(s,r)\geq 2$ et \'{e}ventuellement $m(r,s)=\infty$ si $s\neq r.$
\end{defn}

\begin{defn}
Soit  $(W, \mathcal{S})$ un syst\`{e}me de Coxeter. Nous définissons
l'alg\`{e}bre de Hecke $\mathcal{H} = \mathcal{H}(W, \mathcal{S})$
comme la $\mathbb{Z} [v,v^{-1}]$-alg\`ebre de g\'{e}n\'{e}rateurs $\{
 T_{s}
\}_{s\in \mathcal{S}}$, ceux-ci satisfaisant les relations $$T^{2}_{s}=
v^{-2}+(v^{-2}-1)T_{s}$$ pour tout $s\in \mathcal{S}$ et
$$\underbrace{T_{s}T_{r}T_{s}...}_{m(s,r)\, \mathrm{termes}
}=\underbrace{T_{r}T_{s}T_{r}...}_{m(s,r)\, \mathrm{termes} }$$ si
$s,r \in \mathcal{S}$ et $sr$ est d'ordre $m(s,r)$.

Si
$x=s_{1}s_{2}\cdots s_{n}$ est une expression r\'{e}duite de $x$, on
d\'{e}finit $T_x=T_{s_1}T_{s_2}\cdots T_{s_n}$ ($T_x$ ne d\'epend pas
 du choix de la d\'ecomposition r\'eduite). Nous posons $q=v^{-2}.$ Nous pouvons montrer que $\{T_x\}_{x\in W}$ est une base
de $\mathcal{H}$ sur $\mathbb{Z}[v,v^{-1}].$

\end{defn}

Soit $\mathcal{T} \subseteq W$ le sous-ensemble des r\'eflexions,
c'est \`{a} dire, tous les \'{e}l\'{e}ments qui sont conjugu\'{e}s
aux \'{e}l\'{e}ments de $\mathcal{S}$.

\begin{defn}
Soit $k$ un corps  de
 caract\'eristique diff\'erente de $2$.
Une repr\'esentation de dimension finie de $W$ sur $k$ est
appel\'{e}e r\'eflexion fid\`{e}le (RF) si elle est fid\`{e}le et
 si l'ensemble d'\'{e}l\'{e}ments de $W$ qui ont un espace de points fixes de
codimension un coincide avec l'ensemble des r\'eflexions.
\end{defn}

\begin{defn}
Pour chaque objet gradu\'{e} $M=\bigoplus_i M_i,$ et chaque entier $n$,
 on
 d\'{e}finit l'objet
d\'{e}cal\'{e} $M(n)$ par $(M(n))_i=M_{i+n}.$
\end{defn}

\begin{defn}
Soit $\tau :\mathcal{H}\rightarrow \mathbb{Z} [v,v^{-1}]$
  l'application
d\'{e}finie par
$$\tau\left(\sum_{x\in W}p_xT_x\right) =p_1 \,\,\,\,\,\,\,\,\,\,\,\,
 (p_x\in \mathbb{Z} [v,v^{-1}]).$$
\end{defn}

\begin{defn}
Pour toute petite cat\'{e}gorie additive $\mathcal{A}$, on
d\'{e}finit le groupe de Grothendieck scind\'{e}
$\langle\mathcal{A}\rangle$. C'est le groupe libre sur les objets
de $\mathcal{A}$ modulo les relations $M=M'+M''$ chaque fois que
$M\cong M'\oplus M''$. Chaque objet $A\in \mathcal{A}$ d\'{e}finit
un \'{e}l\'{e}ment $\langle A \rangle \, \in \, \langle \mathcal{A}
\rangle$.
\end{defn}

\begin{defn}\label{pallo} Soit $U$ une représentation de $W$ sur le corps $k$. Soit
$R=S(U^{*})= R(U)$ l'algèbre symétrique de $U^*$, c'est à
 dire
l'algèbre des fonctions régulières sur $U$, sur laquelle $W$
agit par fonctorialité. L'alg\`ebre $R$ est gradu\'ee de la
 manière  suivante : $R=\bigoplus_{i\in
\mathbb{Z}}R_i$ avec $R_2 = U^*$ et $R_i=0$ pour $i$ impair. Nous notons
 $R^s$ le sous-anneau de $R$ des invariants pour
l'action de $s\in W$. Pour $s\in \mathcal{S}$ nous définissons
 $\theta_s=R\otimes_{R^s}R$.

 La catégorie de Soergel $\mathbf{B}_k(U)$ associée à $U$ est la
 catégorie des
 \newline $(R,R)-$bimodules $\mathbb{Z}$-gradués, dont les objets sont
 les facteurs directs des
 sommes
 directes finies d'objets du type $\theta_{s_1}\otimes_R\cdots \otimes_R\theta_{s_n}(d)$,
 pour un $d\in \mathbb{Z}$ et $s_1,\ldots, s_n \in \mathcal{S}$. Par la suite nous noterons  $\theta_{s_1}\cdots \theta_{s_n}(d)$ le $(R,R)$-bimodule $\theta_{s_1}\otimes_R\cdots \otimes_R\theta_{s_n}(d)$
\end{defn}

\begin{defn} 
 Nous disons qu'une paire $(U,U')$ dont $U$ est une
  représentation et $U'$ une sous-représentation de $W$ est une bonne paire, si elle satisfait à la propriété que
 que les réflexions simples agissent comme des reflexions, et $U$ satisfait aussi que
 le
 corollaire 4.2 de \cite{L} est valable pour $R=R(U)$, c'est à dire :
\begin{propri}\label{cuac}
Nous définissons les entiers  $n_i$  par $\tau((1+T_{s_1})\ldots
(1+T_{s_n})) =\sum_in_iq^i$. Alors, il existe un isomorphisme de
$R-$modules \`{a} droite gradu\'{e}s
$$\mathrm{Hom}(\theta_{s_1}\cdots\theta_{s_n},R) \simeq \oplus_i n_i
 R(2i). $$
 \end{propri}
\end{defn}

\begin{rem}\label{trop}
 Dans l'article \cite{S} Soergel montre que la propriété \ref{cuac} est vraie si $U$ est RF. Dans le même article Soergel construit une
 représentation réelle RF $U_0$ pour chaque système de Coxeter
 $(W,\mathcal{S})$.
 Cette représentation admet une sous-représentation $U'_0$ isomorphe
 à la représentation géométrique. Donc $(U_0,U'_0)$ est une bonne paire.
\end{rem}

\subsection{}\label{top} Nous  fixerons jusqu'à la fin de cet article une bonne paire ($V,V'$).
Soit   $R'=R(V')$
 l'algèbre des fonctions régulières sur $V'$. L'inclusion $V'\subset V$
 induit une surjection $Q : R\rightarrow R'$, ce qui permet de voir
 $R'$
 comme $R$-module. 
Pour $s\in \mathcal{S}$ nous définissons 
 $\theta'_s=R'\otimes_{R^{'s}}R'$. 
 
 Soit $\mathbf{B}=\mathbf{B}_k(V)$, $ \mathbf{B'}=\mathbf{B}_k(V')$,
 $\mathbf{C}$ la
 catégorie de $(R,R)-$bimodules, $\mathbf{C}'$ la
 catégorie de $(R',R')-$bimodules et $\mathfrak{X} :
 \mathbf{B}\rightarrow
 \mathbf{C}'$ le foncteur additif qui envoie $M$ vers
 $R'\otimes_RM\otimes_RR'.$

\section{Théorèmes principaux}
Nous avons que $s\in \mathcal{S}$ agit comme une réflexion dans $V$ et
 dans $V'$, donc nous pouvons trouver
 $V''$ stable par $s$, avec $V=V'\oplus V''$. Comme $R(V)=R(V')\otimes
 R(V'')$, nous avons $R^s=R^{'s}\otimes R(V'').$ Cette dernière
 égalité
 nous permet d'obtenir les isomorphismes suivantes dans $\mathbf{C}'$ :
\begin{equation}\label{ecu}
 \mathfrak{X}(\theta_s)\simeq R'\otimes_{R^s}R' \simeq \theta'_s 
\end{equation}
Nous avons besoin du lemme suivant pour expliciter ce que veut dire le titre
 de cet article.

\begin{lem}\label{pasa}
Soit $M\in \mathbf{B}$. Alors $\mathfrak{X}(M)\simeq R'\otimes_RM$ comme
 $(R',R)-$bimodules. 
\end{lem}
\begin{proof}
Il suffit de le prouver pour $M=\theta_{s_1}\cdots\theta_{s_k}$, avec
 $s_1,\ldots,s_k\in \mathcal{S}$. 
Comme $\mathrm{ker}(Q)=V^{'\bot}\cdot R$ (ici $V^{'\bot}$ est
 l'ensemble
 des formes linéaires sur $V$ nulles sur $V'$), il suffit de montrer
 que
 $(R'\otimes_RM)\cdot V^{'\bot}=0$. Mais si $s\in \mathcal{S}$, alors
 $s$
 agit
 trivialement sur $V/V'$, alors $W$ agit trivialement sur $V/V'$, donc
 aussi sur $V^{'\bot}\simeq (V/V')^*$. Nous concluons que $V^{'\bot}\subset
 R^W$, alors $(R'\otimes_RM)\cdot V^{'\bot}=V^{'\bot}\cdot(R'\otimes_RM)
 =0$
 ce qui permet de conclure.
\end{proof}

Avec ce lemme nous voyons aisément les isomorphismes dans $\mathbf{C}'$ :
\begin{equation}
 \mathfrak{X}(\theta_{s_1}\cdots \theta_{s_k})\simeq \theta'_{s_1}\cdots
 \theta'_{s_k} \simeq \mathfrak{X}(\theta_{s_1})\cdots\mathfrak{X}(
 \theta_{s_k})
\end{equation}
Ces isomorphismes généralisent l'isomorphisme (\ref{ecu}) et montrent
 qu'on peut regarder
 $\mathfrak{X}$ comme un foncteur (tensoriel) de \textbf{B} vers
 $\mathbf{B'}$.
Les trois théorèmes suivants expliquent le fait qu'on considère que la
 représentation dans $V$ est équivalente à la représentation dans $V'$
 dans la théorie de
 Soergel.

\begin{thm}\label{1}
Pour tout $M,N\in \mathbf{B}$, le morphisme canonique :
$$R'\otimes_R \mathrm{Hom}_{\mathbf{C}}(M,N)\xrightarrow{\sim}
 \mathrm{Hom}_{\mathbf{C'}}(\mathfrak{X}(M),\mathfrak{X}(N)) $$ est un
 isomorphisme de $R'$-modules gradués.
\end{thm}

\begin{thm}\label{2}
$M$ est indécomposable dans \textbf{B} si et seulement si
 $\mathfrak{X}(M)$ est indécomposable dans $\mathbf{B'}$.
\end{thm}

\begin{thm}\label{ult}
$\mathfrak{X}$ induit un isomorphisme au niveau des groupes de
 Grothendieck scindés, qu'on appelle aussi $\mathfrak{X} :\left\langle
 \mathbf{B} \right\rangle \rightarrow \left\langle \mathbf{B'}
 \right\rangle $ 
\end{thm}

\subsection{} 
\begin{defn}
 Une représentation est appelée  ``reflection vector faithful'' (RVF) si elle satisfait que les réflexions agissent
 comme des réflexions et que les différentes réflexions ont des
 différents $(-1)$-espaces propres 
\end{defn}

\begin{rem}
L'ensemble des représentations RF est contenu dans l'ensemble des représentations RVF. La représentation géométrique d'un groupe de Coxeter $W$ est RVF mais non pas nécessairement RF, comme le montre l'exemple du groupe diédral infini.
\end{rem}

\subsection{} Soit $U$ une représentation RVF. Dans le théorème 1.10 de l'article \cite{S1}, Soergel donne un isomorphisme d'anneaux entre l'algèbre de Hecke de $W$ et le groupe de Grothendieck scindé de $\mathbf{B}_k(U)$, $\varepsilon : \mathcal{H}
 \xrightarrow{\sim}\left\langle
 \mathbf{B}_k(U) \right\rangle $.  Nous posons la conjecture de Soergel sur
 $\mathbf{B}_{\mathbb{R}}(U)$ :
\begin{conj}[Soergel]\label{cs}
Soit $U$ une représentation RF de $W$ sur $\mathbb{R}$. Pour tout $x\in W$, il existe un $R(U)$-bimodule indécomposable
$\mathbb{Z}$-gradu\'{e} $B_x\in \mathbf{B}_k(U)$ tel que
 $\varepsilon(C'_x)=
<B_x>$, o\`u $C'_x$ est l'\'el\'ement de la base de Kazhdan-Lusztig
 associ\'e \`a $x$.
\end{conj}

\begin{rem}\label{nimp}
Dans \cite{S}, Soergel montre que prouver \ref{cs} pour un $U$ quelconque satisfaisant les hypothèses de \ref{cs} implique la
conjecture de positivit\'{e} des polyn\^{o}mes de Kazhdan-Lusztig.
\end{rem}

\begin{rem}
La conjecture \ref{cs} peut se généraliser pour $k$ un corps infini. Dans ce cas c'est connu qu'elle n'est plus vraie en toute généralité. Cependant, 
dans \cite{S2}
 Soergel montre que si la
caractéristique de $k$ est plus grande que le nombre de Coxeter de
 $W$ et si $W$ est un groupe de Weyl
fini, alors la conjecture \ref{cs} est équivalente \`{a} une
 partie
de la conjecture de Lusztig portant sur les caractères des
représentations irréductibles de groupes algébriques sur
$k$ (par exemple $\mathrm{GL}_n(\mathbb{\bar{F}}_p)$). 
\end{rem}

\begin{rem}\label{croc}
Si $V$ et $V'$ (les représentations qu'on a fixé dans la section \ref{top}) sont RVF, les théorèmes  \ref{2} et \ref{ult} impliquent que la conjecture de
 Soergel sur
 \textbf{B} est équivalente à la conjecture de Soergel sur
 $\mathbf{B'}$ . 

En particulier, le remarques \ref{trop} et \ref{nimp} impliquent que quand $k=\mathbb{R}$, si nous démontrons la conjecture \ref{cs} pour $U=U'_0$ (la représentation géométrique
définie dans la remarque \ref{trop}), alors nous démontrons \ref{cs} pour $U=U_0$, et donc nous prouvons la conjecture de positivité de Kazhdan-Lusztig.
\end{rem}

\section{Travail sur les corps de fractions de $R$ et de $R'$}
Pour démontrer ces théorèmes, nous commençons par un lemme important :

\begin{lem}\label{8}
Soit $(s_1,\ldots,s_k)\in \mathcal{S}^k$ et
 $M=\theta_{s_1}\cdots\theta_{s_k}$. Ils existent un entier $n$  et
 une application surjective $f\in
 \mathrm{Hom}_{\mathbf{C}}(M,R^{n})$, tels que le morphisme
 $\mathrm{Hom}_{\mathbf{C}}(R^{n},R)\rightarrow
 \mathrm{Hom}_{\mathbf{C}}(M,R)$ qui
 s'en déduit est un isomorphisme de $R$-modules \`a droite. 
\end{lem}

\begin{rem}
 Ce morphisme peut \^etre regard\'e comme la projection $ \Gamma_{\geq 0}M \rightarrow \Gamma_{\geq 0}M/\Gamma_{> 0}M$ dans la notation de l'article \cite{S} section 5. 
\end{rem}

\begin{proof}
\'Etant donné que la propriété \ref{cuac} est valable pour $R=R(V)$,
 dans l'article \cite{L} nous montrons qu'il existe une base  $f_1,\ldots,
 f_r\in \mathrm{Hom}_{\mathbf{C}}(M,R)$ comme $R$-module, appelée base
 des feuilles légères, et qu'il existe un ensemble
 $\{x_1,\ldots, x_r\}\subseteq M$ tel que $f_i(x_j)=0$ si $i<j$ et
 $f_i(x_i)=1.$
 Ceci permet de conclure que  $n=r$ et  $f=\sum_i f_i$  satisfont les
 propriétés du théorème.
\end{proof}

Nous continuons avec les notations du lemme \ref{8}. 
 Le lemme \ref{8} nous donne une suite exacte de $(R,R)-$bimodules
\begin{equation}\label{sui}0\rightarrow \mathrm{ker}f \rightarrow M
 \rightarrow R^n\rightarrow
 0. \end{equation}
Cette suite est scindée comme suite de $R-$modules à gauche, $R^n$
 étant
 projectif. Nous obtenons donc une suite exacte de $(R',R)-$bimodules
$$0\rightarrow R'\otimes_R\mathrm{ker}f \rightarrow R'\otimes_R
 M\rightarrow R'\otimes_RR^n\rightarrow 0$$
Comme $M\in \mathbf{B}$, par le lemme \ref{pasa} l'action à droite de
 $R$ sur $R'\otimes_R M$ se factorise par $R'$,
et comme  $ R'\otimes_R\mathrm{ker}f$ s'injecte dans $R'\otimes_R M$,
 l'action à droite de $R$ sur $R'\otimes_R \mathrm{ker}f$ 
se factorise aussi par $R'$, donc nous pouvons considérer $
 R'\otimes_R\mathrm{ker}f$
 comme un $(R',R')-$bimodule.
Finalement nous obtenons une suite exacte de $R'-$modules
\begin{equation}\label{exacte}
0 \rightarrow \mathrm{Hom}_{\mathbf{C}'}(R^{'n},R')\rightarrow
 \mathrm{Hom}_{\mathbf{C}'}(R'\otimes_R M,R')\rightarrow
 \mathrm{Hom}_{\mathbf{C}'}(R'\otimes_R\mathrm{ker}f,R') 
 \end{equation}

\begin{prop}\label{cra}
$\mathrm{Hom}_{\mathbf{C}'}(R'\otimes_R\mathrm{ker}f,R')=0$
\end{prop}
\begin{proof}
Avant de prouver cette proposition il nous faut prouver deux lemmes

\begin{lem}\label{14}
Soit $K'$ le corps de fractions de $R'$ et $M\in \mathbf{B}$. Nous avons
 un isomorphisme de $(K',R)-$bimodules :

$ K'\otimes_R M \simeq K'\otimes_R M\otimes_R K' $
\end{lem}
\begin{proof}
 Il suffit de  prouver  l'isomorphisme
 pour $M=\theta_s.$ Pour ceci il faut commencer par prouver
 l'isomorphisme de $(K',R)-$bimodules suivant:
\begin{equation}\label{eco}
K'\otimes_{R^{'s}}R' \simeq K'\otimes_{R^{'s}}K' 
\end{equation}
Soit $in: R' \hookrightarrow K'$ l'injection canonique, et soit $x'_s$
 l'équation de l'hyperplan défini par $s$ dans $V'$. Alors le morphisme
 
$Id\otimes in : K'\otimes_{R^{'s}}R' \rightarrow K'\otimes_{R^{'s}}K'$
 a pour inverse le morphisme 
$$ \frac{p_1}{q_1}\otimes \frac{p_2}{q_2} \mapsto
  \frac{p_1}{q_1q_2s(q_2)}\otimes s(q_2)p_2  $$
ce qui montre la formule (\ref{eco}).

Nous avons alors une suite d'isomorphismes de $(K',R)-$bimodules:
\begin{displaymath}
\begin{array}{llll}
 K'\otimes_{R^s}R&\simeq& K'\otimes_{R'}(R'\otimes_R(R\otimes_{R^s}R))&
 \\
 &\simeq& K'\otimes_{R'}(R'\otimes_{R^s}R') & (\mathrm{ lemme } \:\:
 \ref{pasa}) \\
 &\simeq& K'\otimes_{R'}(R'\otimes_{R^{'s}}R') & (\ref{ecu}) \\
 &\simeq&  K'\otimes_{R^{'s}}K' & (\mathrm{ isomorphisme }
 \:\:(\ref{eco})
 \\
 &\simeq&  K'\otimes_{R^{s}}K'&.
\end{array}
\end{displaymath}
Donc nous avons montré le lemme pour $M=\theta_s,$ ce qui complète la preuve du
 lemme.

\end{proof}
\begin{rem}
Dans la suite nous allons considérer $ K'\otimes_R M$, via l'isomorphisme du
 lemme \ref{14} comme un $(K',K')-$bimodule.
\end{rem}

\begin{defn}
 Soit $A$ un anneau muni d'une action de $W$. Pour $w\in W$, nous notons $A_w$ le $(A,A)-$bimodule ayant  $A$ comme ensemble sous-jacent, et dont l'action à gauche est l'action habituel mais l'action à droite est tordue par $w$, c'est-à-dire, $a\cdot a'=aw(a')$, pour tout $a,a'\in A.$
\end{defn}

\begin{lem}\label{15}
 Nous utilisons les notations du lemme \ref{8}. Il existe un ensemble
 d'entiers naturels $\{n_w\}_{w\in W}$, et un isomorphisme de
 $(K',K')-$bimodules :
\begin{equation}\label{prima}K'\otimes_{R} M\simeq \bigoplus_{w\in
 W}(K'_w)^{n_w}
\end{equation}
avec $n_1=n$, où $1$ est l'identité de $W$.
\end{lem}
\begin{proof}
Nous avons une suite exacte de $(R,R)-$bimodules :
$$ 0 \rightarrow R_s \xrightarrow{\mu_s}  \theta_s \xrightarrow{m_s} R
 \rightarrow 0  $$
où $m_s$ est le morphisme multiplication et $\mu_s(1)=x_s\otimes 1-
 1\otimes x_s$, où $x_s$ est l'équation dans $V$ de l'hyperplan de
 réflexion de  $s$. Comme $R$ est un $R$-module projectif, en tensorisant
 par $K'$
 sur $R$ nous retrouvons une suite exacte de $(K',K')-$bimodules par le
 lemme \ref{14}  :
\begin{equation}\label{ocho}
  0 \rightarrow K'_s \rightarrow  K'\otimes_{R}\theta_s \rightarrow K'
 \rightarrow 0 
\end{equation}
 
Soit $x'_s$ l'équation dans $V'$ de l'hyperplan de réflexion de  $s$.
 La suite \ref{ocho} est scindée par le morphisme $\nu_s :
 K'\otimes_{R}\theta_s \rightarrow K'_s $ donné par $(K'\otimes_{R^s}R
 \ni a\otimes b
 \mapsto as(b)/{2x'_s})$.  Donc nous avons un isomorphisme de
 $(K',K')-$bimodules
 :
\begin{equation}\label{pri}
K'\otimes_{R'} \theta_s \simeq K'\oplus K'_s.
\end{equation}

Nous avons les isomorphismes de $(K',K')-$bimodules :
 
 \begin{displaymath}
\begin{array}{llll}
 K'\otimes_{R}\theta_{s_1}\cdots \theta_{s_k}&\simeq&
 K'\otimes_{R^{s_1}}K'\otimes_{R^{s_2}}\cdots \otimes_{R^{s_k}}K'&
 (\mathrm{ lemme } \:\:
 (\ref{14}))\\
 &\simeq& (K'\otimes_{R^{s_1}}K')\otimes_{K'}\cdots
 \otimes_{K'}(K'\otimes_{R^{s_k}}K') &  \\
 &\simeq& (K'\oplus K'_{s_1})\otimes_{K'}\cdots \otimes_{K'}(K'\oplus
 K'_{s_k}) & (\mathrm{ equation } \:\: (\ref{pri})).
\end{array}
\end{displaymath}

Donc le fait que $K'_x \otimes_{K'} K'_{y}\simeq K'_{xy}$ permet de
 conclure la première partie du lemme.

Maintenant nous  prouverons que $n_1=n$. Par
 la construction de l'isomorphisme \ref{pri}, si $$l=\mathrm{card}\{
 1\leq
 i_1 <i_2 <\ldots <i_p\leq k  ;  s_{i_1}\cdots s_{i_p}=1\},$$ alors
 $n_1=l$.

En outre,  la propriété
 \ref{cuac} dit que si nous définissons les entiers  $n'_i$  par
 $\tau((1+T_{s_1})\ldots
(1+T_{s_k})) =\sum_in'_iq^i$, alors, il existe un isomorphisme de
$R-$modules \`{a} droite gradu\'{e}s
$$\mathrm{Hom}(\theta_{s_1}\cdots\theta_{s_k},R) \simeq \oplus_i n'_i
 R(2i). $$
Mais par le lemme \ref{8}, ceci implique que $n=\sum_i n'_i.$
La spécialisation de  l'algèbre de Hecke en $q=1$ est un morphisme
 $\rho$ d'algèbres de $\mathcal{H}$ vers $\mathbb{C}W=\oplus_{x\in W}
 \mathbb{C}x$, l'algèbre du groupe de $W$. Nous appliquons $\rho$ des
 deux cotés de l'équation 
$$(1+T_{s_1})\ldots
(1+T_{s_k}) = \sum_in'_iq^i + \sum_{w\neq 1} \lambda_w T_w $$
dont les $\lambda_w$ sont des polynômes en $q$, et nous obtenons
$$ (1+s_1)\cdots (1+s_k)=\sum_i n'_i + \sum_{w \neq 1} \lambda_w(1)
 w$$
Ceci implique que $\sum_i n'_i=l$, et ceci finit la preuve du lemme.
\end{proof}
\subsection{}
\textbf{Preuve de la proposition \ref{cra} : }
 Dans la suite exacte (\ref{sui}), le fait que $R^n$ est projectif
 comme $R$-module à gauche nous permet de tensoriser par $K'$ et obtenir
 encore  une suite exacte de $(K',K')-$bimodules, par le lemme \ref{14} :
$$0\rightarrow K'\otimes_R \mathrm{ker}f \rightarrow K'\otimes_R M
  \rightarrow K^{'n} \rightarrow 0$$
de par le lemme \ref{15} cette suite est isomorphe à :
$$0\rightarrow K'\otimes_R \mathrm{ker}f \rightarrow
 K^{'n}\oplus(\bigoplus_{w\neq 1}(K'_w)^{n_w}) \rightarrow K^{'n}
 \rightarrow 0$$
Il est facile de voir que
\begin{equation}\label{pro}
\mathrm{Hom}_{K',K'}(K'_w,K')\simeq \begin{cases} 0 \,\,\,\, \text{ si }
 \,\, w\neq 1 \\   K \text{ si } \,\, w=1
\end{cases}
\end{equation}
Donc nous pouvons conclure que 
$$ K'\otimes_R \mathrm{ker}f\simeq \bigoplus_{w\neq 1}(K'_w)^{n_w}.$$
Cet isomorphisme et (\ref{pro}) permettent de conclure que 
 $\mathrm{Hom}_{K',K'}(K'\otimes_R\mathrm{ker}f,K')=0$.

Supposons que $g\in
 \mathrm{Hom}_{\mathbf{C}'}(R'\otimes_R\mathrm{ker}f,R')$, et $g\neq
 0.$ Alors $0\neq Id\otimes g\in
 \mathrm{Hom}_{K',K'}(K'\otimes_R\mathrm{ker}f,K')$, ce qui est une
 contradiction et permet
 de finir la preuve de la proposition.

\end{proof}
\section{Preuves des théorèmes}
\textbf{Preuve du théorème \ref{1} : }
Comme conséquence de la proposition \ref{cra} et de la suite exacte
 (\ref{exacte}), nous concluons que 
$$\mathrm{Hom}_{\mathbf{C'}}(R^{'n},R')\simeq
 \mathrm{Hom}_{\mathbf{C'}}(R'\otimes_R M,R'), $$
 
 ce qui démontre le théorème \ref{1} pour
 $M=\theta_{s_1}\cdots\theta_{s_k}$ et $N=R$. 
 
Par le lemme 3.3 de \cite{L}, nous savons que si $M,N \in \mathbf{B}$, le
 morphisme 
\begin{displaymath}
\begin{array}{lll}
\mathfrak{F}_s(M,N) :
 \mathrm{Hom}(\theta_s
M,N) &\rightarrow &  \mathrm{Hom}(M, \theta_s N)(2) \\
\ \ \ \ \ \ \ \ \ \ \ \ f &\mapsto & (m
\mapsto x_s\otimes f(1\otimes m)+ 1\otimes f(1\otimes x_s m))     
\end{array}
\end{displaymath}
 est un isomorphisme de $R$-modules \`a droite gradués. Nous connaissons
 explicitement son inverse : si $g\in \mathrm{Hom}(M, \theta_s N)(2)$, on
 peut écrire de
 manière
unique $g(m)=1 \otimes g_1(m) +x_s \otimes g_2(m)$, avec $g_1(m),
g_2(m)\in N$. Ceci définit les morphismes $g_1$ et $g_2$
associ\'{e}s \`{a} $g$. La fonction inverse de $\mathfrak{F}_s(M,N)$
 est $\mathfrak{G}_s(M,N) : \mathrm{Hom}(M, \theta_s
 N)(2)
\to \mathrm{Hom}(\theta_s M,N)$, le morphisme qui envoie $g$ vers le
 morphisme
$\lambda \otimes m \mapsto \lambda g_2(m)$, avec $\lambda \in R$ et
 $m\in M$.
 
 Nous avons de même un isomorphisme de $R'$-modules \`a droite gradu\'es.
$$\mathfrak{F}'_s(M',N') :
 \mathrm{Hom}_{\mathbf{C'}}(\theta'_sM',N')\simeq
 \mathrm{Hom}_{\mathbf{C'}}(M',\theta'_sN')(2). $$.

Le lemme suivante découle directement des définitions des morphismes
 impliqués.

\begin{lem}Nous avons le diagramme commutatif suivant :
 $$ 
  \shorthandoff{;:!?}
  \xymatrix{
{R'\otimes_R
 \mathrm{Hom}_{\mathbf{C}}(M,\theta_sN)(2)}\ar[r]\ar[d]_{Id\otimes
 \mathfrak{G}_s(M,N)}&{\mathrm{Hom}_{\mathbf{C'}}(\mathfrak{X}(M),\theta'_s\mathfrak{X}(N))(2)}\\
{R'\otimes_R
 \mathrm{Hom}_{\mathbf{C}}(\theta_sM,N)}\ar[r]&{\mathrm{Hom}_{\mathbf{C'}}(\theta'_s\mathfrak{X}(M),\mathfrak{X}(N))}\ar[u]_{Id\otimes
 \mathfrak{F}'_s(\mathfrak{X}(M),\mathfrak{X}(N))}
}
  $$
où les morphismes horizontaux sont les morphismes naturels.
\end{lem}

Ce lemme démontre le théorème (\ref{1}) pour $M$ et $N$ des
 bimodules basiques (de la forme $\theta_{s_1}\cdots\theta_{s_k}(d)$). Et
 ceci nous donne la preuve pour $M,N\in \mathbf{B}.$ $\hfill\Box$
\bigskip

\textbf{Preuve du théorème \ref{2}}. 
 
Nous commençons par démontrer la partie "si" du théorème. Nous  commencerons par montrer que si  $M\in \mathbf{B}$ et $M\neq 0$, alors $\mathfrak{X}(M)\neq 0$. Nous savons que $\mathfrak{X}(M) \simeq (V^{'\bot} R \cdot M)\backslash M$ (on rappelle que $V^{'\bot}$ est
 l'ensemble
 des formes linéaires sur $V$ nulles sur $V'$), donc il suffit de montrer que $V^{'\bot} R \cdot M\neq M$. Soit $M=\oplus_{i\geq k}M_i\oplus \{0\}$ son écriture graduée, avec $M_k\neq 0$. Comme les éléments non nuls de $V^{'\bot}$ sont de degré $2$ (voir la définition \ref{pallo}), alors les degrés des éléments non nuls de  $V^{'\bot} R \cdot M$ sont supérieurs à $k$, ce qui nous permet de conclure  que $V^{'\bot} R \cdot M\neq M$.

Si $M$ est décomposable, il existent $M_1,M_2\neq 0$ avec $M\simeq M_1\oplus M_2$. Ceci implique que $\mathfrak{X}(M)\simeq \mathfrak{X}(M_1)\oplus \mathfrak{X}(M_2)$, avec $\mathfrak{X}(M_1),\mathfrak{X}(M_2)\neq 0$ par ce qu'on vient de voir, donc $\mathfrak{X}(M)$ est décomposable. Ceci implique la partie "si" du théorème.

  Donc nous nous intéressons à la partie "seulement si". 
Soit $ I$ un ensemble de représentants des classes d'isomorphismes des bimodules indécomposables  de
 \textbf{B}. Nous fixons $M\in I$ jusqu'\`a la fin de cette preuve. Par le théorème de Krull-Schmidt il existe une suite $s_1,\ldots, s_n \in \mathcal{S}$ et un $d\in \mathbb{Z}$ tels que $M$  est un
 facteur
 direct de $\theta_{s_1}\cdots \theta_{s_k}(d)$

\'Etant donné que $E=\mathrm{End}_{\mathbf{C}}(\theta_{s_1}\cdots
 \theta_{s_k})(d)$ possède une base (finie) comme $R-$module (la base des
 feuilles légères), il est facile de voir que $E_0$, le $\mathbb{C}-$sous
 espace
 vectoriel de $E$ formé par les endomorphismes de degré zéro, est une
 $\mathbb{C}-$algèbre de dimension finie. Ceci implique que si
 $G=\mathrm{End}_{\mathbf{C}}(M)$, alors $G_0$ (endomorphismes de degré zéro) est
 aussi une
 $\mathbb{C}-$algèbre de dimension finie.

Si $G'=\mathrm{End}_{\mathbf{C'}}(\mathfrak{X}(M))$ nous avons un morphisme
entre des  $\mathbb{C}-$algèbres de dimension finie $G_0\rightarrow G'_0$
qui est surjectif comme conséquence du théorème \ref{1}. Donc nous pouvons relever les idempotents de $G'_0$ en des
 idempotents de $G_0$. Ceci implique 
que si $\mathfrak{X}(M)$ est décomposable, alors $G'_0$ a des
 idempotents non triviaux, et donc $G_0$ aussi, donc $M$ est
 décomposable 
ce qui est absurde. $\hfill \Box$

\textbf{Preuve du théorème \ref{ult}}. Nous savons  que
 $\mathfrak{X}$ étant un foncteur additif, il définit bien par passage
 au
 quotient un morphisme de groupes entre les groupes de Grothendieck
 scindés.

\begin{lem}\label{esto}
Soient $M,M'\in I$. Alors $\mathfrak{X}(M)\simeq \mathfrak{X}(M')\Rightarrow M=M'$
\end{lem}
 \begin{proof}
 Supposons  $\mathfrak{X}(M)\simeq \mathfrak{X}(M')$.
 Soit $f :\mathfrak{X}(M)\rightarrow \mathfrak{X}(M')$ un
 isomorphisme,
 et $g$ son inverse. Soient $F:M\rightarrow M'$  et
 $G:M'\rightarrow M$ des relevés respectifs, c'est-à-dire, tels que $Id\otimes F=f$ et $Id\otimes G=
 g.$ Nous posons $E=
 \mathrm{End}_{\mathbf{C}}(M),$ 
 $E'=\mathrm{End}_{\mathbf{C'}}(\mathfrak{X}(M))$ et nous notons $E_0$, $E'_0$ leurs parties de degré zéro respectives. 

Comme conséquence du théorème \ref{1} nous avons un morphisme surjectif $E_0 \rightarrow E'_0$ de $\mathbb{C}-$algèbres de dimension finie.
Comme $E_0$ est une algèbre locale de dimension finie sur $\mathbb{C}$
 et
comme $E'_0$ est un quotient non nul de $E_0$, un element de $E_0$ est
inversible si et seulement si son image dans $E'_0$ est inversible. On en
deduit que $G \circ F$ est inversible. De meme, $F\circ G$ est inversible,
donc $F$ et $G$ sont des isomorphismes. \end{proof}

 Par le théorème de Krull-Schmidt (cf remarque 1.3 de \cite{S}) nous savons que $\left\langle
 N\right\rangle=\left\langle M\right\rangle \in \left\langle
 \mathbf{B'}\right\rangle \Leftrightarrow N\simeq M \in \mathbf{B'}$.

Soit $M\in I$. Par le théorème \ref{2}, le bimodule $\mathfrak{X}(M)$ est indécomposable, donc le théorème
 de
 Krull-Schmidt et le lemme \ref{esto} permettent de conclure que $\{\left\langle
 \mathfrak{X}(M)\right\rangle\}_{M\in I}$ est libre comme
 $\mathbb{Z}-$module dans
 $\left\langle \mathbf{B'} \right\rangle$.  Le fait que $\{\left\langle
 M\right\rangle\}_{M\in I}$ est une base du  $\mathbb{Z}-$module
  $\left\langle \mathbf{B}\right\rangle$  permet de conclure que
 $\mathfrak{X}:\left\langle \mathbf{B}\right\rangle \rightarrow
 \left\langle
 \mathbf{B'}\right\rangle$
est injectif. 
 
 La surjectivité de $\mathfrak{X}$ se déduit du lemme suivant :
 \begin{lem}\label{co}Le foncteur
$\mathfrak{X}: \mathbf{B} \rightarrow \mathbf{B'}$
est essentiellement surjectif. 
\end{lem} 
\begin{proof}
 Chaque objet indécomposable $\gamma'$ de $\mathbf{B'}$ est un facteur direct d'un objet
 $X'=\theta'_{s_1}\cdots \theta'_{s_p}(k)$. Soit $X=\theta_{s_1}\cdots \theta_{s_p}(k)$.  Soit $ E= \mathrm{End}_{\mathbf{C}}(X)$, et
 $E_0$ sa partie gradué de degré zéro. Soit $ E'=
 \mathrm{End}_{\mathbf{C'}}(X')$, et $E'_0$ sa partie
 gradué de degré zéro. Comme
 $E_0$ est une $\mathbb{C}-$algèbre de dimension finie, et le
 morphisme $E_0\rightarrow E'_0$ est surjectif comme conséquence du
 théorème
 \ref{1}, alors tout idempotent de $E'_0$ peut se relever à un
 idempotent de $E_0$. En particulier l'idempotent définissant $\gamma'$, ce
 qui
 permet de compléter les preuves du lemme \ref{co} et du théorème \ref{ult}. 
 \end{proof}

\end{document}